%% file: chacon-weak-limits.tex
\documentclass{amsart}

  \usepackage[utf8]{inputenc}
  \usepackage{amsmath}
  \usepackage{amssymb}
  \usepackage{amsthm}
  \usepackage{mathrsfs}
  \usepackage{bm}
  \usepackage{bbm}
  \usepackage{graphicx}
  \usepackage{color}

	\newtheoremstyle{slanted}% name
	{}%      Space above, empty = `usual value'
	{}%      Space below
	{\slshape}% Body font
	{}%         Indent amount (empty = no indent, \parindent = para indent)
	{\bfseries}% Thm head font
	{.}%        Punctuation after thm head
	{ }% Space after thm head: \newline = linebreak
	{}%         Thm head spec
	
	\theoremstyle{slanted}
	\newtheorem{theo}{Theorem}[section]
	\newtheorem{prop}[theo]{Proposition}

	\newtheorem{lemma}[theo]{Lemma}
	
	\newtheorem{corollary}[theo]{Corollary}

	\def\egdef{:=}
	\def\Id{\mathop{\mbox{Id}}}
	
	\newcommand{\tend}[3][]{\xrightarrow[#2\to#3]{#1}}
		
	\newcommand{\EE}{\mathbb{E}}

	\def\ind#1{\mathbbmss{1}_{#1}}
	\newcommand{\ZZ}{\mathbb{Z}}

	\renewcommand{\L}{\mathscr{L}}
	\newcommand{\A}{\mathscr{A}}

\def\order{\mbox{order}}

\title{Weak limits of powers of Chacon's automorphism}

\author{É. Janvresse}
\author{A. A. Prikhod'ko}
\author{T. de la Rue}
\author{V. V. Ryzhikov}

\address{É. Janvresse, T. de la Rue:
Laboratoire de Math\'ematiques Rapha\"el Salem,
Normandie Université, Universit\'e de Rouen, CNRS --
Avenue de l'Universit\'e --
76801 Saint \'Etienne du Rouvray, France.}
\email{elise.janvresse@univ-rouen.fr\\Thierry.de-la-Rue@univ-rouen.fr}

\address{A. A. Prikhod'ko, V. V. Ryzhikov:
Moscow State University, Faculty of Mechanics and Mathematics,
Leninskie Gory, Moscow, 119991 Russia.}
\email{sasha.prihodko@gmail.com\\vryzh@mail.ru}
\thanks{Part of this research was done during the visit of the second author at Rouen University on a CNRS visiting researcher position.}
\begin{document}

\bibliographystyle{amsplain}

\begin{abstract}
  We completely describe the weak closure of the powers of the Koopman operator associated to Chacon's classical automorphism. We show that weak limits of these powers are the ortho-projector to constants and an explicit family of polynomials. As a consequence, we answer negatively the question of $\alpha$-weak mixing for Chacon's automorphism.
\end{abstract}

\maketitle

\section{Introduction}

The classical version of Chacon's automorphism, which is the main subject of the present work, was described by Friedman in~\cite{Friedman1970}. It is a famous example of a rank-one automorphism, for which we recall briefly the construction by cutting and stacking: We start with a Rokhlin tower of height $h_0\egdef 1$, called Tower $0$. At step $n$, Tower $n-1$ (of height $h_{n-1}$) is cut into 3 sub-columns, a spacer is inserted above the middle column before stacking all parts to get Tower $n$, of height $h_n=3h_{n-1}+1$. 
This transformation, denoted hereafter by $T$ and acting on a standard Borel probability space $(X, \A, \mu)$, is known to present an interesting combination of ergodic and spectral properties. It is weakly mixing but not strongly mixing (see \cite{Chacon1969}, where the historical version of Chacon's automorphism is constructed with only 2 sub-columns, but whose arguments also apply in the classical case).  Del~Junco proved in \cite{DelJunco1978} that $T$ has trivial centralizer, then improved this result by showing with Rahe and Swanson that it has minimal self-joinings~\cite{DJ-R-S1980}. The second and fourth author 
proved 
in~\cite{P-R2000} that the convolution powers of its maximal spectral type are pairwise mutually singular.  Their method involves the identification, in the weak closure of the powers of the associated Koopman operator $\hat T$, of an infinite family of polynomials in $\hat T$.

 An automorphism $S$ is said to be \emph{$\alpha$-weakly mixing} (for some $0\le\alpha\le1$) if there exists a sequence $(k_j)$ of integers such that 
$\hat S^{k_j}$ converges weakly to $\alpha\Theta+(1-\alpha)\Id$, where $\Theta$ is the ortho-projector to constants.
The disjointness of the convolution powers is automatically satisfied in the case of $\alpha$-weakly mixing transformations with $0<\alpha<1$ (see Katok~\cite{Katok} and Stepin~\cite{Stepin1986}).  This property has been applied for  numerous counterexamples  in ergodic theory ~\cite{DJ-L1992}. 
 The question of $\alpha$-weak mixing for Chacon's automorphism is a special case of a  general problem to tell which operators can be obtained as weak limits of powers of $\hat T$. For a recent application of weak limits of powers of the Koopman operator, see \cite{A-L-R2012}. Examples of transformations  with non-trivial explicit weak closure of powers  are given in~\cite{Ryzhikov2012}.

The purpose of the present paper is to completely describe the weak closure 
$$
\L \egdef \mbox{WCl}(\{\hat T^{-k}, k\in\ZZ\})
=\left\{\lim_{j\to\infty} \hat T^{-k_j} \mbox{ for a sequence of integers }(k_j)\right\}.
$$
of the powers of $\hat T$. Our main result, Theorem~\ref{Thm:weak limits}, states that $\L$ is reduced to $\Theta$ and an explicit family of polynomials in $\hat T$. Our result implies in particular that $T$ is not $\alpha$-weakly mixing for any $0<\alpha<1$. Note that partial results in the description of $\L$ have also been given by Ageev \cite{Ageev2003} who gave all polynomials in $\hat T$ of degree at most 1 in $\L$.

An essential ingredient in our description of $\L$ is the identification of particular weak limits, along the sequences $(mh_n)_{n\ge1}$, where $h_n$ is the height of the $n$-th tower in the cutting-and-stacking construction. As observed in~\cite{P-R2012}, these weak limits are given by a family of polynomials $(P_m(\hat T))$. In Section~\ref{sec:integral automorphism}, we give a definition of these polynomials $P_m$ based on the representation of $T$ as an integral automorphism over the 3-adic odometer (this representation was already used in \cite{P-R2000}). We provide inductive formulas for these polynomials in Section~\ref{sec:recurrence}. These formulas enable us to derive useful results about the asymptotic behavior of their coefficients (Section~\ref{sec:properties}). Then, by expanding the integers $(k_j)$ along the heights $(h_n)$, we prove that if the weak limit of $\hat T^{-k_j}$ is not $\Theta$, then it can be factorized by some polynomial $P_m(\hat T)$ (Proposition~\ref{prop:factorization}).

\section{Representation of Chacon's automorphism as integral automorphisms over the 3-adic odometer}
\label{sec:integral automorphism}

\subsection{Definition of the polynomials $\bm{P_m}$ in the 3-adic group}

Consider the compact group of 3-adic numbers
$$
\Gamma\egdef \ZZ_3=\Bigl\{x=(x_0, x_1, x_2, \dots), x_k\in\{0, 1, 2\}\Bigr\}.
$$
We denote by $\lambda$ the Haar measure on $\Gamma$: Under $\lambda$, the coordinates $(x_k)_k$ are i.i.d., uniformly distributed in $\{0, 1, 2\}$.

We introduce two $\lambda$-preserving transformations on $\Gamma$: 
\begin{itemize}
 \item The shift-map $\sigma:x=(x_0, x_1, \dots)\in \Gamma\mapsto \sigma x=(x_1, x_2, \dots)\in\Gamma$.
 \item The adding-machine transformation $S:x\in \Gamma\mapsto x+1 \in\Gamma$, where $1 \in \Gamma$ is identified with the sequence $(1,0,0,\dots)$. (In general, each integer $j$ is identified with an element of $\Gamma$, so that $S^jx=x+j$ for all $j\in\ZZ$ and all $x\in\Gamma$.)
\end{itemize}
We define the cocycle $\phi:\Gamma\setminus\{(2, 2, \dots)\}\rightarrow\ZZ$, where $\phi(x)$ is the first coordinate of $x$ which is different from 2:
$$
\phi(x)\egdef 
\begin{cases}
0& \mbox{if } x=2\dots 20*\\
1& \mbox{if } x=2\dots 21* .
\end{cases}
$$
We set $\phi^{(0)}(x)\egdef0$ and for $m\ge1$, 
$$
\phi^{(m)}(x)\egdef \phi(x)+\phi(Sx)+\dots +\phi(S^{m-1}x).
$$
Let us define $\pi_m$ as the probability distribution of $\phi^{(m)}$ on $\ZZ$: 
$\pi_m(j)\egdef \lambda(\phi^{(m)}=j)$,
and the polynomial $P_m$ by
$$
P_m(X)\egdef\EE_\lambda\left[X^{\phi^{(m)}}\right]=\sum_{j=0}^m \pi_m(j)X^j.
$$
Note that the degree of $P_m$ is strictly less than $m$ as soon as $m>2$.

\medskip

\subsection{Integral automorphisms over the 3-adic odometer}
We will make use of the following representations of Chacon's automorphism.
For each $n\ge0$, we define 
$$
X_n \egdef \left\{ (x,i): \ x\in\Gamma, 0\le i\le h_n-1+\phi(x)\right\}
$$
(see Figure~\ref{fig:Xn}).
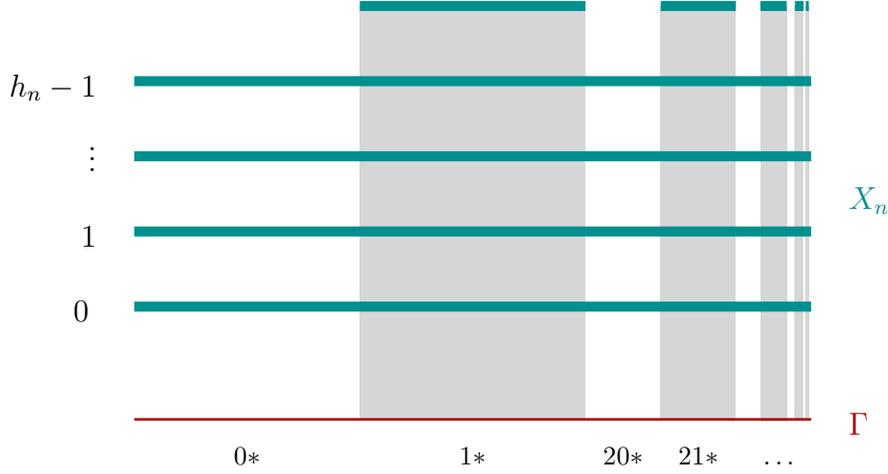
\begin{figure}[ht]
  \centering
  \input{space-X_n.pstex_t}
  \caption{The space $X_n$}
  \label{fig:Xn}
\end{figure}

We consider the transformation $T_n$ of $X_n$, defined by 
$$
T_n(x,i)\egdef \begin{cases}
                 (x, i+1) & \mbox{ if } i+1 \le h_n-1 +\phi(x)\\
                 (Sx, 0) & \mbox{ if } i = h_n-1 +\phi(x).
               \end{cases}
$$
Let us introduce the map $\psi_n:X_n\rightarrow X_{n+1}$ defined by 
$$
\psi_n(x,i)\egdef (\sigma x, x_0h_n +i + \ind{x_0=2}).
$$
Observe that $\psi_n$ is bijective. 
Moreover, it conjugates the transformations $T_n$ and $T_{n+1}$. 
We consider the probability measure $\mu_n$ on $X_n$: For a fixed $i$ and a set $A\subset\{(x, i), x\in\Gamma\}$, 
$$
\mu_n(A)\egdef \frac1{h_n+1/2}\lambda\left(\{x\in\Gamma, (x,i)\in A\}\right).
$$
The transformation $T_n$ preserves $\mu_n$ and the map $\psi_n$ sends $\mu_n$ to $\mu_{n+1}$. It follows that all the measure-preserving dynamical systems $(X_n, T_n, \mu_n)$ are isomorphic. 

For $0\le i\le h_n-1$, we set 
$E_{n, i}\egdef\{(x,i): x\in\Gamma\}\subset X_n$. We have $E_{n, i}=T_n^iE_{n, 0}$, hence 
$$
\left\{E_{n, 0}, \dots, E_{n, h_n-1}\right\}
$$
is a Rokhlin tower of height $h_n$ for $T_n$. 
Moreover, for any $n\ge0$, and any $0\le i\le h_n-1$, 
$$
\psi_n(E_{n,i})= E_{n+1,i}\sqcup E_{n+1, h_n+i}\sqcup E_{n+1, 2h_n+i+1}.
$$
Fix $n_0$. By composition of the isomorphisms $(\psi_n)$, we can view all these Rokhlin towers inside $X_{n_0}$. The above formula shows that the towers are embedded in the same way as the towers of Chacon's automorphism. 
Therefore, $(X_{n_0}, \mu_{n_0}, T_{n_0})$ is isomorphic to $(X, \mu, T)$.

\subsection{Weak limits along subsequences $\bm{(m h_n)_{n\ge1}}$}
\begin{lemma}
 \label{Lemma: uniformite dans la tour}
Let $m\ge1$ and $u\ge0$ be fixed integers. Then 
$$
\sup_{A, B} \left| \mu\left(T^{mh_n+u}B\cap A \right) - \sum_{i\in\ZZ} \pi_m(i+u)\,\mu\left(T^{-i}B\cap A\right) \right| \tend{n}{\infty} 0.
$$
where the supremum is taken over any sets $A$ and $B$ which are union of levels of Tower $n$.
\end{lemma}
\begin{proof}
We may identify $(X, \mu, T)$ with $(X_n,\mu_n, T_n)$ and the level $j$ of Tower $n$ with 
$E_{n, j}=\{(x,j): x\in\Gamma\}\subset X_n$. 

Assume that $A=E_{n, k}$ for $0\le k\le h_n-1$ and $B=E_{n, j}$ for $m\le j\le h_n-1$ . 
We have 
$$
T^{mh_n}B=T^{mh_n}E_{n, j} = \left\{ (S^mx, j-\phi^{(m)}(x)), x\in\Gamma\right\}.
$$
Then, viewed in $X_n$,
$$
T^{mh_n}B\cap A = S^m\left(\left\{x\in\Gamma:\ \phi^{(m)}(x) = j-k\right\}\right)\times\{k\},
$$
and
$$
\mu\left(T^{mh_n}B\cap A \right) = \lambda \left( \phi^{(m)}=j-k\right)\mu(E_{n, k})
= \pi_m(j-k)\mu(E_{n, k}).
$$
Moreover, $\mu\left(T^{-i}B\cap A\right) = \mu(E_{n, k})$ if $k=j-i$, and zero otherwise. 
We obtain that 
$$
\mu\left(T^{mh_n}B\cap A \right) = \sum_{i\in\ZZ} \pi_m(i)\,\mu\left(T^{-i}B\cap A\right).
$$
By additivity of the measure $\mu$, the above equality remains true if $A$ is a union of levels of Tower $n$, and if $B$ is a union of levels $j\ge m$ of Tower $n$.
Finally, removing the restriction on the levels in $B$, we have 
$$
\left|\mu\left(T^{mh_n}B\cap A \right) - \sum_{i\in\ZZ} \pi_m(i)\,\mu\left(T^{-i}B\cap A\right)\right|
\le 2m\, \mu(E_{n, 0})\tend{n}{\infty}0.
$$
This proves the lemma for $u=0$. For an arbitrary $u$, we apply this result to the part of $T^uB$ which remains in the $n$-th tower. We get an extra error term of order $|u|\,\mu(E_{n, 0})$.
\end{proof}

As a direct consequence of Lemma~\ref{Lemma: uniformite dans la tour}, we recover the result from~\cite{P-R2012}:
\begin{theo}
For any $m\ge1$, we have the weak convergence
 \begin{equation}
%   \label{eq:def_P_m}
  \lim_{n\to\infty} \hat T^{-mh_n} = P_m(\hat T).
\end{equation}
\end{theo}

\section{Recurrence formulas for $P_m$}

\label{sec:recurrence}

\subsection{Description of the sequence $\bm{(\phi(S^{j}x))_{j\in\ZZ}$}}
\label{Sec:description of the cocycle}

Let $x\in\Gamma\setminus\{(2, 2, \dots)\}$. We say that $\order (x)=k\ge0$ if $x_0=\dots=x_{k-1}=2$ and $x_{k}\not=2$.

Since the first digit in the sequence $(\dots, x-1, x, x+1, \dots )$ follows a periodic pattern $\dots 012012012\dots$, the contribution of points of order $0$ in the sequence $(\phi(S^{j}x))_{j\in\ZZ}$ provides a periodic sequence of blocks $01$ separated by one symbol given by a point of higher order (see Figure~\ref{fig:structure}).
To fill in the missing symbols corresponding to positions $j$ such that $\order(x+j)\ge1$, we observe that, if $x$ starts with a $2$, then for all $j\in\ZZ$, 
$$
\phi(x+3j)=\phi(\sigma x+j).
$$
Hence the missing symbols are given by the sequence $(\phi(S^{j}\sigma x))_{j\in\ZZ}$.

\begin{figure}[h]
  \begin{center}
\begin{tabular}{lllllllllllllllll}
0 & 1 & . & 0 & 1 & . & 0 & 1 & . & 0 & 1 & . & 0 & 1 & . & 0 & $\leftarrow$ contribution of order 0\\
  &   & 0 &   &   & 1 &   &   & . &   &   & 0 &   &   & 1 &   & $\leftarrow$  contribution of order 1\\
  &   &   &   &   &   &   &   & 0 &   &   &   &   &   &   &   & $\leftarrow$  higher orders \\
0 & 1 & 0 & 0 & 1 & 1 & 0 & 1 & 0 & 0 & 1 & 0 & 0 & 1 & 1 & 0 & $\leftarrow$  the whole sequence
\end{tabular}
\end{center}
\caption{Structure of the sequence  $(\phi(S^{j} x))_{j\in\ZZ}$.}
\label{fig:structure}
\end{figure}

\begin{lemma}
  \label{lemma:distribution 0 and 1}
  The probability distributions of the random sequences $(\phi(S^{j}x))_{j\in\ZZ}$ and $(1-\phi(S^{-j} x))_{j\in\ZZ}$ are the same.
\end{lemma}
\begin{proof}
  This is an easy consequence of the above construction of the sequence $(\phi(S^{j}x))_{j\in\ZZ}$ and 
  the fact that $\sigma$ preserves the measure~$\lambda$.
\end{proof}

\begin{lemma}
\label{lemma:symmetry of the polynomials}
The coefficients of the polynomial $P_m$ are symmetrical: For all $0\le j\le m$,
$\pi_m(j) = \pi_m(m-j)$.
\end{lemma}
\begin{proof}
The coefficient $\pi_m(m-j)$ is equal to the probability to see $(m-j)$ digits equal to $1$ when looking at $m$ consecutive terms of the sequence $(\phi(S^{j}x))_{j\in\ZZ}$. 
Thus, $\pi_m(m-j)$ is also equal to the probability to see $j$ digits equal to $0$ when looking at $m$ consecutive terms. Using Lemma~\ref{lemma:distribution 0 and 1}, we conclude that $\pi_m(m-j)=\pi_m(j)$.
\end{proof}

\subsection{Recurrence formulas for $\bm{P_m}$}

\begin{theo}
\label{Th:recurrence for P_m}
For all $m\ge0$, 
 \begin{align*}
  P_{3m} (X) &=X^m P_m (X)  ; \\
  P_{3m+1} (X) &=\dfrac{1}{3}X^m\Bigl((1+X)P_m  (X) + P_{m+1} (X) \Bigr) ;\\
  P_{3m+2} (X) &=\dfrac{1}{3}X^m\Bigl(XP_m (X)  + (1+X)P_{m+1} (X) \Bigr). 
\end{align*}
\end{theo}

\begin{proof}
Let $x\in\Gamma\setminus\{(2, 2, \dots)\}$. In the computation of $\phi^{(3m)}(x)$, the contribution of the $2m$ points $x+j$ ($0\le j\le 3m-1$) of order 0 is always $m$.
Because of the structure of the sequence $\left(\phi(S^jx)\right)_{j\in\ZZ}$ described in section~\ref{Sec:description of the cocycle}, the contribution of the other $m$ points is $\phi^{(m)}(\sigma x)$. Hence, 
\begin{equation}
\label{Eq:phi0(3m)}
 \phi^{(3m)}(x) = m+\phi^{(m)}(\sigma x),
\end{equation}
which proves that $P_{3m} =X^m P_m$.

Let us compute $\phi^{(3m+1)}(x)$:
If $x_0=0$, then $\phi^{(3m+1)}(x) = \phi^{(3m)}(x+1)$, which is equal to 
$m+\phi^{(m)}(\sigma (x+1))$ by~\eqref{Eq:phi0(3m)}. 
Hence, $\phi^{(3m+1)}(x) = m+\phi^{(m)}(\sigma x)$.
If $x_0=1$, then $\phi^{(3m+1)}(x) = 1+\phi^{(3m)}(x+1)$. Hence,
$$
\phi^{(3m+1)}(x) =1+m+\phi^{(m)}(\sigma (x+1))=1+m+\phi^{(m)}(\sigma x).
$$
If $x_0=2$, then $\phi^{(3m+1)}(x) = \phi(x) +\phi^{(3m)}(x+1)$. Since $\phi(x)=\phi(\sigma x)$, we get 
$$
\phi^{(3m+1)}(x) = \phi(\sigma x) + m+\phi^{(m)}(\sigma x+1)=m+\phi^{(m+1)}(\sigma x).
$$
\begin{equation}
\label{Eq:phi0(3m+1)}
\phi^{(3m+1)}(x) = 
\begin{cases}
m+ \phi^{(m)}(x)& \mbox{if } x_0=0,\\
1+m+ \phi^{(m)}(\sigma x) & \mbox{if } x_0=1,\\
m+\phi^{(m+1)}(\sigma x)  & \mbox{if } x_0=2.
\end{cases}
\end{equation}
Since each digit appears with probability $1/3$ in first position, and since the distribution of $\sigma x$ conditioned on the first digit is $\lambda$, we get 
$$
P_{3m+1} (X) = \frac1{3}\left( \EE_\lambda\left[X^{m+\phi^{(m)}}\right] + \EE_\lambda\left[X^{m+1+\phi^{(m)}}\right]+ \EE_\lambda\left[X^{m+\phi^{(m+1)}}\right]\right).
$$
This yields the desired formula for $P_{3m+1}$.

In the same way, we compute $\phi^{(3m+2)}(x)$:
\begin{equation}
\label{Eq:phi0(3m+2)}
\phi^{(3m+2)}(x) = 
\begin{cases}
1+m+ \phi^{(m)}(\sigma x)& \mbox{if } x_0=0,\\
1+m+ \phi^{(m+1)}(\sigma x) & \mbox{if } x_0=1,\\
m+\phi^{(m+1)}(\sigma x)  & \mbox{if } x_0=2,
\end{cases}
\end{equation}
and we obtain the recurrence formula for $P_{3m+2}$.
\end{proof}

\subsection{Recurrence formulas for reduced polynomials}
For $m\ge0$, let $\ell(m)$ be the highest power of $X$ dividing $P_m$, so that 
$$ P_m (X)= X^{\ell(m)} \tilde P_m(X), $$
where $\tilde P_m$ is the reduced polynomial of order $m$, $\tilde P_m(0)\neq 0$.

Observe that $\ell(m)$ is the minimum value taken by the cocycle $\phi^{(m)}$. Hence it is easy to see that $m\mapsto\ell(m)$ is non-decreasing, and that 
$$s_m\egdef\ell(m+1)-\ell(m)\in\{0,1\}.$$
Moreover, by Lemma~\ref{lemma:symmetry of the polynomials}, 
\begin{equation}
\label{Eq:relation degree}
\ell(m)+\mbox{deg}(P_m)=m.
\end{equation}
%$\ell(m)$ is equal to $m$ minus the maximum number of $j\in\{1, \dots, m\}$ such that $\phi(x+j)=0$. 
% On the other hand, the degree deg$(P_m)$ of $P_m$ is the maximum value taken by the cocycle $\phi^{(m)}$, that is the maximum number of $j\in\{1, \dots, m\}$ such that $\phi(x+j)=1$. 
Thanks to the recurrence formulas for $P_m$ (see Theorem~\ref{Th:recurrence for P_m}), we deduce recurrence formulas for $\ell(m)$: 
\begin{align}
\label{Eq:recurrence for l(m)}
  \ell(3m)  &= m+ \ell(m) , \notag\\
  \ell(3m+1)  &= m+ \ell(m) , \\
  \ell(3m+2)  &= m+ \ell(m+1).\notag
\end{align}
We also get recurrence formulas for  the reduced polynomials:

\begin{prop}
\label{Prop:recurrence for the reduced polynomials}
Let $m\ge0$. then
\begin{align*}
\tilde P_{3m}(X) & =\tilde P_{m}(X) ; \\
3\tilde P_{3m+1}(X) & =(1+X)\tilde P_{m}(X) + X^{s_m}\tilde P_{m+1}(X) ; \\
3\tilde P_{3m+2}(X) & =X^{1-s_m}\tilde P_{m}(X) + (1+X)\tilde P_{m+1}(X), 
\end{align*}
\end{prop}

\begin{lemma}
  \label{lemma:computation of s_m}
Consider the 3-expansion of $m$: $m=\sum_{j\ge0} m_j 3^j$, where $m_j\in\{0, 1, 2\}$.
Let $i\egdef  \inf\{j:\ m_j\not=1\}$. 
Then $s_m=1$ if $m_i=2$ and $s_m=0$ if $m_i=0$. 
\end{lemma}

\begin{proof}
  By~\eqref{Eq:recurrence for l(m)}, we see that $s_{3m}=0$ and $s_{3m+2}=1$. Moreover, 
  $s_{3m+1}=s_m$.
\end{proof}

\subsection{Degree of $\bm{\tilde P_m}$}
\label{Sec:degree}

Let $d_m$ be the degree of $\tilde P_{m}$: Using~\eqref{Eq:relation degree}, we get $d_m=\mbox{deg}(P_m)-\ell(m) =m-2\ell(m)$. 
Hence, we easily check that 
\begin{equation}
\label{eq:dif d_m}
  d_{m+1}-d_m=1-2s_m\in\{-1,1\},
\end{equation}
and $s_m=\ind{\{d_m>d_{m+1}\}}$.
By Proposition~\ref{Prop:recurrence for the reduced polynomials}, we have
\begin{equation}
  \label{eq:induction for d_m}d_{3m}=d_m, \quad d_{3m+1}=d_m+1, \quad d_{3m+2}=d_{m+1}+1 \ \mbox{ and }\ d_{3m+3}=d_{m+1}.
\end{equation}

We thus get an algorithm to compute the degree $d_m$ of $\tilde P_m$.
Consider the 3-expansion of $m$: $m=\sum_{j\ge0} m_j 3^j$. 
In this expansion, remove all 1's and count the number of blocks of 2's.
Then $d_m$ is equal to the number of removed 1's plus twice the number of 2-blocks. 

Example: Consider $m$ whose expansion in base 3 is $212202$. Remove one $1$, you get $22202$ (two 2-blocks). Hence $d_m=1+2\times2=5$.

\begin{corollary}
The first appearance of a reduced polynomial of degree $d$ is observed at 
$$
m=%(2,\underbrace{1,\ldots,1}_{d-2 \mbox{ terms}})_3
2+\sum_{j=1}^{d-2}3^j= \frac{3^{d-1}+1}{2}.
$$
\end{corollary}

\section{Properties of the probability distribution $\pi_m$}

\label{sec:properties}

\subsection{Unimodality of the distribution $\bm{\pi_m}$}
\label{Sec:profile of the polynomials}

We set $b_j^{(m)}\egdef \pi_m(j+\ell(m))=\lambda\left(\phi^{(m)}=j+\ell(m)\right)$, so that 
$$\tilde P_m(X)=\sum_{j=0}^{d_m} b_j^{(m)}X^j .$$
Observe that by Lemma~\ref{lemma:symmetry of the polynomials}, the coefficients of $\tilde P_m$ are symmetric: 
$b_j^{(m)}=b_{d_m-j}^{(m)}$ for all $0\le j\le d_m$.

\begin{lemma}
 \label{lemma:coeff croissants}
For any $m\ge0$, the coefficients $\left(b_j^{(m)}\right)_{0\le j\le [d_m/2]}$ of $\tilde P_m$ are increasing.
\end{lemma}

\begin{proof}
The lemma holds for $\tilde P_0=1$ and $\tilde P_1(X)=(1+X)/2$. 
Using Proposition~\ref{Prop:recurrence for the reduced polynomials}, we prove the lemma by induction on $m$. 
Assume the property we want to prove is satisfied for some $m$ and $m+1$ and let us prove it is also true for $3m, 3m+1, 3m+2$ and $3m+3$. It obviously holds for $3m$ and $3m+3$ since $\tilde P_{3m}=\tilde P_m$ and $\tilde P_{3m+3}=\tilde P_{m+1}$.
We can assume without loss of generality that $d_m<d_{m+1}$ (that is $s_m=0$). Indeed, the recurrence formulas have the form 
$$
\left(\tilde P_{3m},\tilde P_{3m+1},\tilde P_{3m+2},\tilde P_{3m+3}\right) = F\left(\tilde P_{m},\tilde P_{m+1}\right),
$$
where $F$ is such that 
$$
\left(\tilde P_{3m+3},\tilde P_{3m+2},\tilde P_{3m+1},\tilde P_{3m}\right) = F\left(\tilde P_{m+1},\tilde P_{m}\right).
$$
By Proposition~\ref{Prop:recurrence for the reduced polynomials}, we have 
$3b_0^{(3m+1)} = b_0^{(m)}+b_{0}^{(m+1)}$ and for $j\ge1$ 
$$
3 b_j^{(3m+1)} = b_j^{(m)} + b_{j-1}^{(m)} + b_{j}^{(m+1)}.
$$
Recall that $d_{3m+1}=d_m+1$ (see Section~\ref{Sec:degree}) and we assumed that $d_{m+1}=d_m+1$. Hence, if $d_m$ is even, we have 
$[d_{3m+1}/2]=[d_m/2]=[d_{m+1}/2]$ and the three terms on the RHS of the above equation are increasing functions of $j\in\{0, \dots, [d_{3m+1}/2]\}$. 
If $d_m$ is odd, only the first term on the RHS may not be increasing for the largest value of $j$. But in this case, because of the symmetry of the coefficients (see Lemma~\ref{lemma:symmetry of the polynomials}), we have 
$b_{[d_m/2]+1}^{(m)} =b_{[d_m/2]}^{(m)}$.
This proves that the property holds for $\tilde P_{3m+1}$.  A similar argument proves the property for $\tilde P_{3m+2}$.
\end{proof}

As a consequence of Lemma~\ref{lemma:symmetry of the polynomials} and  Lemma~\ref{lemma:coeff croissants}, we obtain:
\begin{prop}
\label{Prop: symmetry and unimodality}
 For all $m\ge1$, the probability distribution  $\pi_m$ is symmetric and unimodal.
\end{prop}

\subsection{Asymptotic behavior of $\bm{\pi_m}$ when $\bm{d_m\to\infty}$}

Recall that, for all $m$, the coefficients of the polynomial $P_m$ are given by the probability distribution $\pi_m$ on $\ZZ$ which is symmetric and unimodal (see Proposition~\ref{Prop: symmetry and unimodality}). 
Recall that $d_m$ is the degree of the reduced polynomial $\tilde P_m$, that is $(d_m+1)$ is the number of nonzero coefficients of $P_m$.

For any $m\ge1$, consider the Fourier transform $\widehat \pi_m$ defined by 
$$\widehat \pi_{m}(z) \egdef \sum_j \pi_m(j)z^{-j}\quad \forall z\in S^1.$$
Observe that $\widehat \pi_{m}(z) = P_m(1/z)=P_m(z)z^{-m}$.
Moreover, we recover $\pi_m(j) $ by the inverse Fourier transform 
$$
\pi_m(j)=\int_{S^1} z^{j}\widehat \pi_{m}(z) \, dz.
$$

\begin{lemma}
\label{Lemma: coeff polynome}
$$\sup_{j\in\ZZ}\pi_m(j) \tend{d_m}{\infty}0.$$
\end{lemma}

\begin{proof}
Observe that for any $j\in\ZZ$, 
$$
\pi_m(j) \le \int_{S^1} |\widehat \pi_{m}(z)| \, dz =\int_{S^1} |P_{m}(z)| \, dz  .
$$
By Theorem~\ref{Th:recurrence for P_m}, we have for any $m\ge3$: 
If $m$ is a multiple of 3, then $|P_m(z)| = |P_{m/3}(z)|$. Otherwise, there exist $m'<m$ and $m''<m$ such that
$$
|P_{m}(z)|\le \frac{1+|1+z|}{3}\ \sup\left( |P_{m'}(z)|, |P_{m''}(z)| \right), 
$$
Moreover, by~\eqref{eq:induction for d_m}, 
$$
  d_{m'}\ge d_m-2, \quad\mbox{and } d_{m''}\ge d_m-2.
$$
Since $|P_1(z)|\le1$ and $|P_2(z)|\le1$, we easily prove by induction on $m$ that 
\begin{equation}
\label{Eq:Fourier}
  \forall m\ge1, \quad |P_{m}(z)|\le |\alpha(z)|^{(d_m-2)/2},
\end{equation}
where $\alpha(z)\egdef\frac{1+|1+z|}{3}$. 
Since $|\alpha(z)|<1$ for $z\not=1$, we conclude the proof of the lemma.
\end{proof}

\section{Weak limits of powers of $\hat T$}

Recall that $\L$ is the set of all weak limits of powers of $\hat T$, and that $\Theta$ is the ortho-projector to constants.
The purpose of this section is to prove the following Theorem:
\begin{theo}
\label{Thm:weak limits}
%  Let $L\in \L$. Then either $L=\Theta$ or $L$ is of the form 
% $$
% L=P_{m_1}(\hat T)\dots P_{m_r}(\hat T)\hat T^n,
% $$
% for $r\ge0$, $1\le m_1\le \dots\le m_r$ and $n\in\ZZ$.
$$ \L = \{\Theta\}\cup \{ P_{m_1}(\hat T)\dots P_{m_r}(\hat T)\,\hat T^n,\ r\ge0,\ 1\le m_1\le\cdots\le m_r,\ n\in\ZZ\}. $$
Moreover, limits of the form $\hat T^n$ for some $n\in\ZZ$ can only be obtained as limits of $\hat T^{-k_j}$ for \emph{bounded} sequences $(k_j)$.
\end{theo}

\subsection{Correspondence between elements of $\L$ and measures on $\ZZ$}
Any $L\in\L$ is associated with a self-joining $\rho$ by 
$$
\rho(A\times B)\egdef \langle L\ind{A}, \ind{B}\rangle, \quad\forall A, B.
$$
Since $T$ has minimal self-joinings, $\rho$ is of the form 
$$
\rho_\nu\egdef \sum_{j\in\ZZ} \nu(j) \Delta_{-j} + (1-\nu(\ZZ)) \mu\otimes\mu,
$$
where $\Delta_{-j}(A\times B)\egdef \mu(A\cap T^{-j}B)$ and $\nu$ is a positive measure on $\ZZ$ with $\nu(\ZZ)\le 1$. 
Therefore, $L$ has the form 
$$
L= L_\nu\egdef \sum_{j\in\ZZ}\nu(j)\hat T^j +   (1-\nu(\ZZ))\Theta.
$$

We set, for any positive measure $\nu$ on $\ZZ$ with $\nu(\ZZ)\le 1$,
$$
\delta (\nu)\egdef \sum_{j\in\ZZ}|\nu(j+1)-\nu(j)|.
$$

\begin{lemma}
  \label{Lemma:majoration de delta de convolution}
  For any positive measures $\nu$ and $\nu'$ on $\ZZ$ of total mass $\le1$, we have $\delta (\nu\ast\nu')\le \delta(\nu)$.
\end{lemma}
\begin{proof}
\begin{align*}
   \delta (\nu\ast\nu') 
  &= \sum_{j\in\ZZ}|\nu\ast\nu'(j+1)-\nu\ast\nu'(j)|\\
  &= \sum_{j\in\ZZ}\left| \sum_k \Bigl(\nu(j+1-k)-\nu(j-k)\Bigr) \nu'(k) \right|\\
  &\le \sum_k \nu'(k) \sum_{j\in\ZZ}| \nu(j+1-k)-\nu(j-k)|
  \le \delta (\nu).
\end{align*}

\end{proof}

\begin{lemma}
\label{Lemma:convergence to theta}
  Let $(\nu_\ell)_\ell$ be a sequence of positive measures  with $\nu_\ell(\ZZ)\le 1$, such that 
  $\delta(\nu_\ell)\to 0$ as $\ell\to \infty$. Then 
  $L_{\nu_\ell}\to \Theta$.
\end{lemma}
\begin{proof}
For any $A, B$, we have
$$\rho_{\nu_\ell}(A\times B) =\sum_j \nu_{\ell}(j)\, \mu(A\cap T^{-j}B) + (1-\nu_\ell(\ZZ))\mu(A)\mu(B)$$ and 
$$ \rho_{\nu_\ell}(A\times TB) = \sum_j \nu_{\ell}(j+1) \, \mu(A\cap T^{-j}B)+ (1-\nu_\ell(\ZZ))\mu(A)\mu(B).$$
Hence, 
$$
|\rho_{\nu_\ell}(A\times B) - \rho_{\nu_\ell}(A\times TB)|
\le \sum_j |\nu_\ell(j+1)-\nu_{\ell}(j)|\,  \mu(A\cap T^{-j}B) 
\le \delta(\nu_\ell).
$$
It follows that any self-joining $\rho$ which is a limit of a subsequence of $(\rho_{\nu_\ell})$ satisfies:
%By compacity, we can assume that $\rho_{\nu_\ell}$ converges to some self-joining $\rho$.
 $$
\forall A, B\quad  \rho(A\times B)= \rho(A\times TB).
 $$
 By ergodicity of $T$, we get that $\rho=\mu\otimes\mu$ and this proves the convergence of $L_{\nu_\ell}$ to~$\Theta$.
\end{proof}

\begin{lemma}
\label{Lemma: delta de convolees}
$$
\sup_{1\le m_1\le \dots\le m_r}\delta \left( \pi_{m_1}\ast\dots \ast\pi_{m_r}\right) \tend{r}{\infty} 0.
$$
\end{lemma}

\begin{proof}
Let us fix $(m_i)_{1\le i\le r}$ larger than 1. 
Using the fact that the convolution of symmetric unimodal distributions remains symmetric and unimodal (see~\cite{Purkayastha1998}), we obtain by Proposition~\ref{Prop: symmetry and unimodality} that $\pi_{m_1}\ast\dots \ast\pi_{m_r}$ is symmetric and unimodal. Thus
$$
\delta \left( \pi_{m_1}\ast\dots \ast\pi_{m_r} \right) \le 2\ \sup_{j\in\ZZ} \pi_{m_1}\ast\dots \ast\pi_{m_r}(j).
$$
Moreover, by~\eqref{Eq:Fourier}, we have for all $j$
$$
\left| \pi_{m_1}\ast\dots \ast\pi_{m_r}(j)\right| 
= \left| \int_{S^1} z^j \prod_{i=1}^r \widehat\pi_{m_i}(z) \, dz\right|  
\le \int_{S^1} \beta(z)^r\, dz ,
$$
where $\beta(z)\egdef \sup\Bigl(|P_1(z)|, |P_2(z)|,\alpha(z)  \Bigr)$. 
Since $\beta(z) <1$ if $z\not=1$, this ends the proof of the lemma.
\end{proof}

\subsection{Factorization in $\L$}

\begin{lemma}
\label{Lemma:suite reduite}
 Let $(k_j)$ be a sequence of integers such that $k_j=mh_{n_j}+k_j'$, where $k'_j/h_{n_j}\to 0$ and 
$  \lim_{j\to\infty} \hat T^{-k'_j} = L'$. 
Then 
$$\lim_{j\to\infty} \hat T^{-k_j} = P_m(\hat T) L'.$$
\end{lemma}

\begin{proof}
  Let $A$ and $B$ be unions of levels of a fixed tower. For $j$ large enough, $A$ and $B$ are stilll unions of levels in Tower $n_j$, and there exists $A_j$, union of levels in Tower $n_j$, such that 
  $$ \mu(T^{-k'_j}A\vartriangle  A_j) \le |k'_j|/h_{n_j}. $$
  Then we have
  \begin{align*}
    \langle \ind{A},\hat T^{-k_j}\ind{B} \rangle & = \mu (T^{-k'_j}A\cap T^{mh_{n_j}}B) \\
    & = \mu (A_j\cap T^{mh_{n_j}}B) + O(k'_j/h_{n_j})
  \end{align*}
 Fix $\varepsilon>0$. By Lemma~\ref{Lemma: uniformite dans la tour}, for $j$ large enough, $\mu (A_j\cap T^{mh_{n_j}}B)$ is within $\varepsilon$ of $\sum_{i\in\ZZ}\pi_m(i) \mu(T^{-i}B\cap A_j)$. The latter expression
 is equal, up to a correction of order $k'_j/h_{n_j}$, to
 $$ \sum_{i\in\ZZ}\pi_m(i) \mu(T^{-i}B\cap T^{-k'_j}A) = \langle \ind{A}, \hat T^{-k'_j}P_m(\hat T)\ind{B}\rangle, $$
 which converges to $\langle \ind{A}, L'P_m(\hat T)\ind{B}\rangle$ as $j\to\infty$.
\end{proof}

Let $L\in\L$: There exists a sequence $(k_j)$ of integers such that 
$$
  \lim_{j\to\infty} \hat T^{-k_j} = L.
$$
If the sequence $(k_j)$ is bounded, then $L$ is of the form $\hat T^{n}$ for some $n\in\ZZ$.
Otherwise, without loss of generality, we can assume that $k_j$ is positive and $k_j\to +\infty$.

Recall that the heights $(h_n)$ of the Rokhlin towers satisfy: $h_{n+1}=3 h_n+1$. 
We decompose $k_j$ by the greedy algorithm along the integers $(h_n)$:
$$
k_j = \alpha_{0}^j h_{n_j} + \alpha_{1}^j h_{n_j-1}+ \dots + \alpha_{n_j}^j h_{0},
$$
where $\alpha_0^j\not=0$, $0\le \alpha_\ell^j\le 3$ for all $0\le \ell\le n_j$.
Observe that if $\alpha_\ell^j = 3$, then $\alpha_s^j=0$ for all $s>\ell$. 

Using a diagonal procedure to extract a subsequence if necessary, we can suppose that for all $\ell$, 
$\alpha_\ell^j\to\alpha_\ell$ as $j$ goes to $\infty$. 
We have $0\le \alpha_\ell\le 3$, $\alpha_0\not=0$ and if $\alpha_\ell = 3$, then $\alpha_s=0$ for all $s>\ell$. 

\begin{prop}
\label{prop:factorization}
Let $L\in\L$, and let $(k_j)$ and $(\alpha_\ell)$ be as above.
If there exists $r$ such that $\alpha_\ell=2$ for all $\ell>r$, or $\alpha_\ell=0$ for all $\ell>r$, then there exist $m\ge1$ and a sequence $(k'_j)$ such that 
$$
\lim_{j\to\infty} \hat T^{-k_j} = P_m(\hat T)\, L',
$$
where $L'= \lim_{j\to\infty} \hat T^{-k_j'}$. % and $|k_j'|\ll h_{(n_j-r)}$.

If there exist infinitely many $\ell$'s such that $\alpha_\ell\not=2$ and  infinitely many $\ell$'s such that $\alpha_\ell\not=0$, then $L =\Theta$.
\end{prop}
\begin{proof}
Since $h_{n+s}=3^s(h_{n}+1/2)-1/2$, for all $n,s\ge0$, note that for all $r\ge 0$ and all $j$ large enough,
\begin{equation}
  \label{eq:changement de variables}
  \sum_{\ell=0}^r\,\alpha_\ell  h_{n_j-\ell} = m_r\,h_{(n_j-r)} + u_r,
\end{equation}
where 
$$ m_r\egdef \alpha_{0}3^r + \alpha_{1}3^{r-1}+ \dots + \alpha_{r}   \quad\mbox{and}\quad   u_r\egdef \sum_{\ell=0}^{r-1}\,\alpha_\ell h_{r-\ell-1}. $$

First assume that $\alpha_\ell=0$ for all $\ell> r$. Then for $j$ large enough
$$
k_j = \alpha_{0} h_{n_j} + \alpha_{1} h_{n_j-1}+ \dots + \alpha_{r} h_{(n_j-r)} + k_j'
$$
where $k_j'\ll h_{(n_j-r)}$.
By \eqref{eq:changement de variables}, we can rewrite $k_j$ as
$$
k_j= m_r h_{(n_j-r)} + k_j'+u_r,
$$
where $k_j'+u_r \ll h_{(n_j-r)}$. Extracting a subsequence if necessary, we can assume that $\hat T^{-(k_j'+u_r)}$ converges to some $L'$. We conclude using Lemma~\ref{Lemma:suite reduite}.

Assume now that $\alpha_\ell=2$ for all $\ell> r$. Then for $j$ large enough
$$
k_j = \alpha_{0} h_{n_j} + \alpha_{1} h_{n_j-1}+ \dots + \alpha_{r} h_{(n_j-r)} + 2 \sum_{\ell=1}^{r'}  h_{n_j-r-\ell} + k_j'
$$
where $k_j'< h_{n_j-r-r'}$. % and $\alpha_r\not=2$.
Since $2 \sum_{\ell=1}^{r'}  h_{n_j-r-\ell} = h_{(n_j-r)} - h_{n_j-r-r'} - r'$, we get 
$$
k_j = \alpha_{0} h_{n_j} + \alpha_{1} h_{n_j-1}+ \dots + (\alpha_{r}+1) h_{(n_j-r)}  + k_j'' ,
$$
where $|k_j''|=| k_j'- h_{n_j-r-r'} - r'|\ll h_{(n_j-r)}$. We get the conclusion using the same argument as above. This proves the first part of the proposition.

\medskip
Assume now that there exist an infinity of $r$ such that $\alpha_r\not=2$ and  an infinity of $r$ such that $\alpha_r\not=0$. By \eqref{eq:changement de variables}, for all $r\ge0$, and for all $j$ large enough (depending on~$r$),
$$
k_j= m_r h_{(n_j-r)} + u_r + k_j',
$$
where $0\le k_j'< h_{(n_j-r)}$.
We know that $d_{m_r}$ goes to infinity as $r\to\infty$ (see Section~\ref{Sec:degree}) by hypotheses on the sequence $(\alpha_\ell)$. 
By Lemma~\ref{Lemma: coeff polynome}, since $\pi_{m_r}$ is unimodal, we get $\delta(\pi_{m_r}\ast\delta_k)\to 0$ as $r\to\infty$ uniformly with respect to $k\in\ZZ$. 

Let us fix $A_0$ and $B_0$ some sets which are union of levels of a fixed tower, say Tower~$\overline{n}$. 
Fix $\varepsilon>0$. By Lemma~\ref{Lemma:convergence to theta}, $P_{m_r}(\hat T)\,\hat T^k\tend{r}{\infty}\Theta$ uniformly with respect to $k\in\ZZ$.  
Hence, we can find $r$ large enough so that for all $k\in\ZZ$, 
\begin{equation}
  \label{eq:convergence vers le produit}
\left| \sum_{i\in\ZZ} \pi_{m_r}(i+k)\, \mu(A_0\cap T^{-i}B_0) - \mu(A_0)\mu(B_0)\right|<\varepsilon.
\end{equation}
Then, we can choose $j$ large enough to satisfy $\alpha_\ell^j=\alpha_\ell$ for all $\ell\le r$, and
\begin{equation}
  \label{eq:uniformite dans n_j-r}
  \sup_{A, B} \left| \mu\left(A\cap T^{m_rh_{(n_j-r)}+u_r}B\right) - \sum_{i\in\ZZ} \pi_m(i+u_r)\,\mu\left( A\cap T^{-i}B\right) \right| < \varepsilon,
\end{equation}
where the supremum is taken over any sets $A$ and $B$ which are union of levels of Tower $(n_j-r)$ (see Lemma~\ref{Lemma: uniformite dans la tour}). We also assume that $(n_j-r)\ge \overline{n}$, so that $A_0$ and $B_0$ are unions of levels in Tower~$(n_j-r)$.

We want to use~\eqref{eq:uniformite dans n_j-r} in order to estimate
$$ \mu(A_0\cap T^{k_j}B_0) = \mu\left(A_0\cap T^{m_rh_{(n_j-r)}+u_r}(T^{k_j'}B_0)\right). $$
The problem is that, although $B_0$ is a union of levels in Tower $(n_j-r)$, this is not always the case for $T^{k'_j}B_0$. Therefore we cut $B_0$ into 4 disjoint parts: $B_0=B_1\sqcup B_2\sqcup B_3\sqcup B_4$, where
\begin{itemize}
  \item $B_1$ is the part of $B_0$ contained in the first $(h_{(n_j-r)}-k'_j)$ levels of Tower $(n_j-r)$, so that $T^{k'_j}B_1$ is a union of levels in Tower $(n_j-r)$ which is included in $T^{k'_j}B_0$.
  \item $B_2$ is the part of $B_0$ contained in the last $(k'_j-1)$ levels of Tower $(n_j-r)$ which is not under a spacer, so that $T^{k'_j}B_2$ is included in a union of levels in Tower $(n_j-r)$ which is itself included in  $T^{k'_j-h_{(n_j-r)}}B_0$.
  \item $B_3$ is the part of $B_0$ contained in the last $(k'_j-1)$ levels of Tower $(n_j-r)$ which is under a spacer, so that $T^{k'_j}B_3$ is included in a union of levels in Tower $(n_j-r)$ which is itself included in  $T^{k'_j-h_{(n_j-r)}-1}B_0$.
  \item $B_4$ is the part of $B_0$ contained in level $h_{(n_j-r)}-k'_j$ of Tower $(n_j-r)$: $\mu(T^{k'_j}B_4)\le 1/h_{(n_j-r)}$.
\end{itemize}
Using three times~\eqref{eq:uniformite dans n_j-r} and \eqref{eq:convergence vers le produit}, we get
\begin{align*}
  \mu(A_0\cap T^{k_j}B_0) \le & \sum_{i\in\ZZ} \pi_m(i+u_r)\,\mu\left( A_0\cap T^{-i}T^{k'_j}B_0\right)\\
  & + \sum_{i\in\ZZ} \pi_m(i+u_r)\,\mu\left( A_0\cap T^{-i}T^{k'_j-h_{(n_j-r)}}B_0\right)\\
  & + \sum_{i\in\ZZ} \pi_m(i+u_r)\,\mu\left( A_0\cap T^{-i}T^{k'_j-h_{(n_j-r)}-1}B_0\right)\\
  & + 3\varepsilon +1/h_{(n_j-r)}\\
  \le\ & 3 \mu(A_0)\,\mu(B_0) + 6\varepsilon +1/h_{(n_j-r)}.
\end{align*}
Hence, the self joining $\rho$ defined by 
$$
\rho(A\times B)\egdef \lim_{j\to\infty} \mu(A\cap T^{k_j}B) = \langle\ind{A},  L\ind{B}\rangle, \quad\forall A, B,
$$
satisfies $\rho\le 3\mu\otimes\mu$. By ergodicity of $\mu\otimes\mu$, we conclude that $\rho=\mu\otimes\mu$ and $L=\Theta$.
\end{proof}

\begin{proof}[proof of Theorem~\ref{Thm:weak limits}]
First, fix integers $n\in\ZZ$, $r\ge1$ and $1\le m_1\le\cdots\le m_r$ . Using $r$ times Lemma~\ref{Lemma:suite reduite}, we easily construct a sequence $(k_j)$ such that 
$$ \lim_{j\to\infty}\hat T^{-k_j} = \prod_{i=1}^r P_{m_i}(\hat T)\hat T^n. $$

Conversely, we now want to prove that any $L\in \L$ has this form. For $L\in\L$, set 
$$
r_L \egdef 
\sup \{0\}\cup\left\{ r\ge1 :  \exists (m_i)_{i\le r},\exists L'\in \L  \mbox{ s.t. } L=\prod_{i=1}^r P_{m_i}(\hat T) L'\right\}.
$$ 
Observe that if $L=\Theta$, then $r_L=\infty$ (because for all $m$, $P_{m}(\hat T)  \Theta=\Theta$). Conversely, we prove that if $r_L=\infty$, then $L=\Theta$: Suppose that $r_L=\infty$.
Then for any $r$, we can write $L=\prod_{i=1}^r P_{m_i}(\hat T) L'$ with $L'\in\L$. We know that $L$ and $L'$ are of the form $L_\nu$ and $L_{\nu'}$ respectively. Moreover, we have 
$$
\nu=\pi_{m_1}\ast\dots\ast\pi_{m_r}\ast\nu'.
$$
By Lemma~\ref{Lemma:majoration de delta de convolution}, 
$\delta(\nu)\le \delta (\pi_{m_1}\ast\dots\ast\pi_{m_r})$. But the  right-hand side goes to 0 as $r\to\infty$ by Lemma~\ref{Lemma: delta de convolees}, so that $\delta(\nu)=0$. We conclude that $\nu(\ZZ)=0$ and $L=\Theta$.

Assume now that $r_L<\infty$. Let $(k_j)$ be such that $L=\lim_j \hat T^{-k_j}$. If the sequence $(k_j)$ is bounded, then $L=\hat T^n$ for some $n\in\ZZ$, otherwise Proposition~\ref{prop:factorization} applies, and since $L\not=\Theta$ there exists $L'\in\L$  and $m\ge1$ such that $L=P_m(\hat T)L'$. We then have $r_L\ge r_{L'}+1$, and we can prove by induction on $r_L$ that $L$ is of the form
$$ L=\prod_{i=1}^r P_{m_i}(\hat T) \hat T^n. $$
\end{proof}

\bibliography{rank-one}

\end{document}

%% file: space-X_n.pstex_t
\begin{picture}(0,0)%
\includegraphics{space-X_n.pstex}%
\end{picture}%
\setlength{\unitlength}{4144sp}%
\begingroup\makeatletter\ifx\SetFigFont\undefined%
\gdef\SetFigFont#1#2#3#4#5{%
  \reset@font\fontsize{#1}{#2pt}%
  \fontfamily{#3}\fontseries{#4}\fontshape{#5}%
  \selectfont}%
\fi\endgroup%
\begin{picture}(4682,2855)(-391,11088)
\put(4276,11324){\makebox(0,0)[lb]{\smash{{\SetFigFont{12}{14.4}{\rmdefault}{\mddefault}{\updefault}{\color[rgb]{.69,0,0}$\Gamma$}%
}}}}
\put(4276,12674){\makebox(0,0)[lb]{\smash{{\SetFigFont{12}{14.4}{\rmdefault}{\mddefault}{\updefault}{\color[rgb]{0,.56,.56}$X_n$}%
}}}}
\put(-269,11999){\makebox(0,0)[rb]{\smash{{\SetFigFont{12}{14.4}{\rmdefault}{\mddefault}{\updefault}{\color[rgb]{0,0,0}0}%
}}}}
\put(-224,12449){\makebox(0,0)[rb]{\smash{{\SetFigFont{12}{14.4}{\rmdefault}{\mddefault}{\updefault}{\color[rgb]{0,0,0}1}%
}}}}
\put(-224,12899){\makebox(0,0)[rb]{\smash{{\SetFigFont{12}{14.4}{\rmdefault}{\mddefault}{\updefault}{\color[rgb]{0,0,0}$\vdots$}%
}}}}
\put(-224,13349){\makebox(0,0)[rb]{\smash{{\SetFigFont{12}{14.4}{\rmdefault}{\mddefault}{\updefault}{\color[rgb]{0,0,0}$h_n-1$}%
}}}}
\put(676,11144){\makebox(0,0)[b]{\smash{{\SetFigFont{10}{12.0}{\rmdefault}{\mddefault}{\updefault}{\color[rgb]{0,0,0}$0\ast$}%
}}}}
\put(2026,11144){\makebox(0,0)[b]{\smash{{\SetFigFont{10}{12.0}{\rmdefault}{\mddefault}{\updefault}{\color[rgb]{0,0,0}$1\ast$}%
}}}}
\put(2926,11144){\makebox(0,0)[b]{\smash{{\SetFigFont{10}{12.0}{\rmdefault}{\mddefault}{\updefault}{\color[rgb]{0,0,0}$20\ast$}%
}}}}
\put(3376,11144){\makebox(0,0)[b]{\smash{{\SetFigFont{10}{12.0}{\rmdefault}{\mddefault}{\updefault}{\color[rgb]{0,0,0}$21\ast$}%
}}}}
\put(3871,11144){\makebox(0,0)[b]{\smash{{\SetFigFont{10}{12.0}{\rmdefault}{\mddefault}{\updefault}{\color[rgb]{0,0,0}\dots}%
}}}}
\end{picture}%